\documentclass[10pt]{article}
\usepackage[english]{babel}
\usepackage{latexsym,amssymb,amsmath,amsthm,amsfonts}
\usepackage{float,graphicx}

\usepackage{multirow}

\newtheorem{theorem}{Theorem}[section]
\newtheorem{corollary}[theorem]{Corollary}
\newtheorem{conjecture}[theorem]{Conjecture}
\newtheorem{lemma}[theorem]{Lemma}
\newtheorem{proposition}[theorem]{Proposition}

\title{Extremal regular graphs of given chromatic number}

\author{Christian Rubio-Montiel\footnotemark[1]}

\begin{document}
\maketitle

\def\thefootnote{\fnsymbol{footnote}}
\footnotetext[1]{Divisi{\' o}n de Matem{\' a}ticas e Ingenier{\' i}a, FES Acatl{\' a}n, Universidad Nacional Aut{\' o}noma de M{\' e}xico, 53150, Naucalpan, Mexico, {\tt christian.rubio@apolo.acatlan.unam.mx}.}


\begin{abstract}
We define an extremal $(r|\chi)$-graph as an $r$-regular graph with chromatic number $\chi$ of minimum order. We show that the Tur{\' a}n graphs $T_{ak,k}$, the antihole graphs and the graphs $K_k\times K_2$ are extremal in this sense. We also study extremal Cayley $(r|\chi)$-graphs and we exhibit several $(r|\chi)$-graph constructions arising from Tur{\' a}n graphs.
\end{abstract}




\section{Introduction}

An \emph{$r$-regular} graph is a simple finite graph such that each of its vertices has degree $r$. Regular graphs are one of the most studied classes of graphs; especially those with symmetries such as Cayley graphs. Let $\Gamma$ be a finite group and let $X=\{x_1,x_2,\dots,x_t\}$ a generating set for $\Gamma$ such that $X=X^{-1}$ with $1_{\Gamma}\not\in X$; a \emph{Cayley} graph $Cay(\Gamma,X)$ has vertex set consisting of the elements of $\Gamma$ and two vertices $g$ and $h$ are adjacent if $gx_i=h$ for some $1\leq i \leq t$. Cayley graphs are regular but there exist non-Cayley vertex-transitive graphs. The Petersen graph is a classic example of this fact.

The \emph{girth} of a graph is the size of its shortest cycle. An \emph{$(r,g)$-graph} is an $r$-regular graph of girth $g$. An \emph{$(r,g)$-cage} is an $(r,g)$-graph of smallest possible order. The \emph{diameter} of a graph is the largest length between shortest paths of any two vertices. An \emph{$(r;D)$-graph} is an $r$-regular graph of diameter $D$.

While the cage problem asks for the constructions of cages, the \emph{degree-diameter problem} asks for the construction of $(r;D)$-graphs of maximum order. Both of them are open and active problems (see \cite{exoo2013dynamic,miller2013moore}) in which, frequently, it is considered the restriction to Cayley graphs, see \cite{MR3070127,MR2852512}.

In this paper, we study a similar problem using a well-known parameter of coloration instead of girth or diameter. A \emph{$k$-coloring} of a graph $G$ is a partition of its vertices into $k$ independent sets. The \emph{chromatic number $\chi(G)$ of $G$} is the smallest number $k$ for which there exists a $k$-coloring of $G$. 

We define an \emph{$(r|\chi)$-graph} as an $r$-regular graph of chromatic number $\chi$. In this work, we investigate the $(r|\chi)$-graphs of minimum order. We also consider the case of Cayley $(r|\chi)$-graphs.

The remainder of this paper is organized as follows: In Section \ref{section2} we show the existence of $(r|\chi)$-graphs, we define $n(r|\chi)$ as the order of the smallest $(r|\chi)$-graph, and similarly, we define $c(r|\chi)$ as the order of the smallest Cayley $(r|\chi)$-graph. We also exhibit lower and upper bounds on the orders of the extremal graphs. We show that the Tur{\' a}n graphs $T_{ak,k}$, antihole graphs (the complements of cycles) and $K_k\times K_2$ are Cayley $(r|\chi)$-graphs of order $n(r|\chi)$ for some $r$ and $\chi$. To prove that $K_k\times K_2$ are extremal we use instances of the Reed's Conjecture for which it is true. In Section \ref{section3} we only consider non-Cayley graphs. We give another upper bound for $n(r|\chi)$ and we exhibit two families of $(r|\chi)$-graphs with a few number of vertices which are extremal for some values of $r$ and $\chi$. Finally, in Section \ref{section4} we study the small values $2\leq r \leq 10$ and $2\leq \chi \leq 6$. We obtain a full table of extremal $(r|\chi)$-graphs except for the pair $(6|6)$.


\section{Cayley $(r|\chi)$-graphs}\label{section2}

It is known that for any graph $G$, $1\leq\chi(G)\leq \Delta+1$ where $\Delta$ is the maximum degree of $G$.
Therefore, for any $(r|\chi)$-graph we have that \[1\leq \chi \leq r+1.\]

Suppose that $G$ is a $(r|1)$-graph. Hence $G$ is the empty graph, then $r=0$. Therefore, the extremal graph is the trivial graph. We can assume that $2\leq \chi \leq r+1$.

Next, we prove that for any $r$ and $\chi$ such that $2\leq \chi \leq r+1$, there exists a Cayley $(r|\chi)$-graph $G$.

We recall that the \emph{$(n,k)$-Tur{\' a}n graph $T_{n,k}$} is the complete $k$-partite graph on $n$ vertices whose partite sets are as nearly equal in cardinality as possible, i.e., it is formed by partitioning a set of $n=ak+b$ vertices (with $0\leq b < k$) into the partition of independent sets $(V_1,V_2,\dots,V_b,V_{b+1},\dots,V_k)$ with order $|V_i|=a+1$ if $1\leq i \leq b$ and $|V_i|=a$ if $b+1\leq i \leq k$. Every vertex in $V_i$ has degree $a(k-1)+b-1$ for $1\leq i \leq b$ and every vertex in $V_i$ has degree $a(k-1)+b$ for $b+1\leq i \leq k$. The $(n,k)$-Tur{\' a}n graph has chromatic number $k$, and size (see \cite{MR0411988}) \[\left\lfloor \frac{(k-1)n^{2}}{2k}\right\rfloor.\]

\begin{lemma}
The $(ak,k)$-Tur{\' a}n graph $T_{ak,k}$ is a Cayley graph.
\end{lemma}
\begin{proof}
Let $\Gamma$ be the group $\mathbb{Z}_{a}\times\mathbb{Z}_{k}$ and $X=\{(i,j)\colon 0\leq i<a,0<j<k\}$. Then, the graph $Cay(\Gamma,X)$ is isomorphic to $T_{ak,k}$.
\end{proof}

Before to continue, we recall some definitions. Given two graphs $H_1$ and $H_2$, the \emph{cartesian product} $H_1\square H_2$ is defined as the graph with vertex set $V(H_1) \times V(H_2)$ and two vertices $(u,u')$ and $(v,v')$ are adjacent if either $u=v$ and $u'$ is adjacent with $v'$ in $H_2$, or $u'=v'$ and $u$ is adjacent with $v$ in $H_1$. The following proposition appears in \cite{MR2014620}.

\begin{proposition}\label{proposition1}
The cartesian product of two Cayley graphs is a Cayley graph.
\end{proposition}

On the other hand, the chromatic number of $H_1\square H_2$ is the maximum between $\chi(H_1)$ and $\chi(H_2)$, see \cite{MR2450569}. Now we can prove the following theorem.

\begin{theorem}\label{theorem1}
For any $r$ and $\chi$ such that $2\leq \chi \leq r+1$, there exists a Cayley $(r|\chi)$-graph.
\end{theorem}
\begin{proof}
Let $r=a(\chi-1)+b$ where $a\geq 1$ and $0\leq b<\chi-1$. Consider the Cayley graph $H_1=T_{a\chi,\chi}$. The graph $H_1$ has chromatic number $\chi$ and it is an $a(\chi-1)$-regular graph of order $a\chi$.

Additionally, consider the graph $H_2=T_{b+1,b+1}=K_{b+1}$. The graph $H_2$ has chromatic number $b+1 < \chi$ and it is a $b$-regular graph of order $b+1$.

Therefore, the graph $G=H_1\square H_2$ is a Cayley graph by Proposition \ref{proposition1} such that it has chromatic number \[\max\{\chi(H_1),\chi(H_2)\}=\chi,\]
regularity $r$ and order $a\chi(b+1)$.
\end{proof}

Now, we define $n(r|\chi)$ as the order of the smallest $(r|\chi)$-graph and $c(r|\chi)$ as the order of the smallest Cayley $(r|\chi)$-graph. Hence, \[r+1\leq n(r|\chi) \leq c(r|\chi) \leq a\chi(b+1)\] where $r=a(\chi-1)+b$ with $a\geq 1$ and $0\leq b<\chi-1$.

To improve the lower bound we consider the $(n,\chi)$-Tur{\' a}n graph $T_{n,\chi}$. Suppose $G$ is an $(r|\chi)$-graph. Let $\varsigma$ be a $\chi$-coloring of $G$ resulting in the partition $(V_1,V_2,\dots,V_\chi)$ with $|V_i|=a_i$ for $1\leq i \leq \chi$. Then the largest possible size of $G$ occurs when $G$ is a complete $\chi$-partite graph with partite sets $(V_1,V_2,\dots,V_\chi)$ and the cardinalities of these partite sets are as equal as possible. This implies that \[\frac{nr}{2}\leq\left\lfloor \frac{(\chi-1)n^{2}}{2\chi}\right\rfloor \leq \frac{(\chi-1)n^{2}}{2\chi},\] 
since $G$ has size $rn/2$. After some calculations we get that \[\frac{r\chi}{\chi-1}\leq n.\]

\begin{theorem}\label{theorem2}
For any $2\leq \chi \leq r+1 $, \[\left\lceil \frac{r\chi}{\chi-1} \right\rceil \leq n(r|\chi) \leq c(r|\chi) \leq \frac{r-b}{\chi-1}\chi(b+1)\]
where $\chi-1|r-b$ with $0\leq b <\chi-1$.
\end{theorem}

An $(r|\chi)$-graph $G$ of $n(r|\chi)$ vertices is called \emph{extremal $(r|\chi)$-graph}. Similarly, a Cayley $(r|\chi)$-graph $G$ of $c(r|\chi)$ vertices is called \emph{extremal Cayley $(r|\chi)$-graph}. When $\chi-1|r$ the lower bound and the upper bound of Theorem \ref{theorem2} are equal. We have the following corollary.

\begin{corollary}\label{corollary1}
The Cayley graph $T_{a\chi,\chi}$ is an extremal $(a(\chi-1)|\chi)$-graph.
\end{corollary}

In the remainder of this paper we exclusively work with $b\not=0$, that is, when $\chi-1$ is not a divisor of $r$.

\subsection{Antihole graphs}

A \emph{hole graph} is a cycle of length at least four. An \emph{antihole graph} is the complement $G^c$ of a hole graph $G$. Note that a hole graph and its antihole graph are both connected if and only if their orders are at least five. In this subsection we prove that antihole graphs of order $n$ are extremal $(r|\chi)$-graphs for any $n$ at least six. There are two cases depending of the number of vertices.

\begin{enumerate}
\item $G=C^c_{n}$ for $n=2k$ and $k\geq 3$.

The graph $G$ has regularity $r=2k-3$ and chromatic number $\chi=k$. Any $(2k-3|k)$-graph has an even number of vertices and at least $\frac{r\chi}{\chi-1}=\frac{(2k-3)k}{k-1}=2k-\frac{k}{k-1}$ vertices.

If $k>2$, then $\frac{k}{k-1}<2$. Therefore we have the following result:

\[n(2k-3,k)=c(2k-3,k)=2k\] for all $k\geq 3$. 

\item $G=C^c_{n}$ for $n=2k-1$ and $k\geq 4$.

The graph $G$ has regularity $r=2k-4$ and chromatic number $\chi=k$. Any $(2k-4|k)$-graph has at least $\frac{r\chi}{\chi-1}=\frac{(2k-4)k}{k-1}=2k-2-\frac{2}{k-1}$ vertices.

If $k-1>2$, we have that $\frac{2}{k-1}<1$. Therefore \[2k-2\leq n(2k-4,k) \leq c(2k-4,k) \leq 2k-1\] for all $k\geq 4$.

Suppose that $G$ is a $(2k-4|k)$-graph of $2k-2$ vertices. Then $G=((k-1)K_2)^c$, i.e., $G$ is the complement of a matching of $k-1$ edges. Then $\chi(G)=k-1$, a contradiction. Therefore \[n(2k-4,k)=c(2k-4,k)=2k-1\] for all $k\geq 4$.
\end{enumerate}

Therefore, we have the following theorem.

\begin{theorem}
The antihole graphs of order $n\geq 6$ are extremal $(n-3|\left\lceil \frac{n}{2}\right\rceil )$-graphs.
\end{theorem}

A hole graph is also considered a \emph{$2$-factor} since is a spanning $2$-regular graph. For short, we denote the disjoint union of $j$ cycles of lenght $i$ as $jC_i$. 

Let $G$ be an union of cycles \[a_{3}C_{3}\cup a_4C_{4}\cup\ldots\cup a_{2t}C_{2t}\] for $a_i\geq 0$ with $i\in\{3,4,\dots,2t\}$. Note that the complement $G^c$ of $G$ is the join of the complement of cycles. 

\begin{theorem}
The graph $(a_{3}C_{3}\cup a_{4}C_{4}\cup\ldots\cup a_{2t}C_{2t})^c$ is extremal if $a_5+a_7+\dots+a_{2t-1}+1<a_3$.
\end{theorem}	
\begin{proof}
Let $G^c = (a_{3}C_{3}\cup a_{4}C_{4}\cup\ldots\cup a_{2t}C_{2t})^c$. The graph $G^c$ has order $n=3a_3+4a_4+\dots+2ta_{2t}$, regularity $r=n-3$ and chromatic number $\chi=a_3+2a_4+3a_5+3a_6+\dots+ta_{2t-1}+ta_{2t}$ since the the chromatic numbers of $C^c_3$, $C^c_4$, $C^c_5$, ...,$C^c_i$ are $1,2,3,\dots, \left\lceil i/2\right\rceil$ respectively.

Any $(r|\chi)$-graph has at least $\frac{r\chi}{\chi-1}=r+\frac{r}{\chi-1}=n-\frac{3\chi-n}{\chi-1}$ vertices for $r=n-3$. If $\frac{3\chi-n}{\chi-1}<1$ then $G^c$ is extremal, that is, when \[2\chi+1 < n,\]
i.e. when
\[a_5+a_7+\dots+a_{2t-1}+1<a_3.\]

\end{proof}

Moreover, we have the following results.

\begin{theorem}\label{theorem6}
Since $C^c_{n}$ is extremal then
\begin{enumerate}
\item When $n$ is even, if $G=(a_{3}C_{3}\cup a_{4}C_{4}\cup\ldots\cup a_{2t}C_{2t})^c$ is a graph of order $n$ such that $a_5+a_7+\dots+a_{2t-1}=a_3$, then $G$ is extremal.

\item When $n$ is odd, if $G=(a_{3}C_{3}\cup a_{4}C_{4}\cup\ldots\cup a_{2t}C_{2t})^c$ is a graph of order $n$ such that $a_5+a_7+\dots+a_{2t-1}=a_3+1$, then $G$ is extremal.
\end{enumerate}
\end{theorem}

\begin{corollary}
Since the antihole graphs of order $n\geq 8$ are $(r|\chi)$-graphs, then there exist many non-isomorphic extremal $(r|\chi)$-graphs (not necessarily Cayley).
\end{corollary}

For instance, there are three extremal $(5,4)$-graphs, namely, $C^c_8$, $(2C_4)^c$ and $(C_3\cup C_5)^c$. See also Table \ref{tabla1}.

%

\subsection{The case of $r=\chi$}

In this subsection, we discuss the case of $r=\chi=k$, i.e., the $(k|k)$-graphs of minimum order. We have the following bounds so far: \[\left\lceil\frac{k^2}{k-1}\right\rceil=k+1\leq n(k|k)\leq 2k.\]

We prove that the upper bound is correct except for $k=4$ and maybe for $k=6,8,10,12$. To achieve it, we assume that there exist $(k|k)$-graphs of order $n\leq 2k-2$, that is 
\begin{equation}\label{Equation1}
\left\lceil\frac{n}{2}\right\rceil<k=\chi.
\end{equation}
Now, we use a bound for the chromatic number arising from the Reed’s Conjecture, see \cite{MR1610746}. We recall the clique number $\omega(G)$ of a graph $G$ is the largest $k$ for which $G$ has a complete subgraph of order $k$.

\begin{conjecture}
For every graph $G$, \[\chi(G)\leq \left\lceil\frac{\omega(G)+1+\Delta(G)}{2}\right\rceil.\]
\end{conjecture}

It is known that the conjecture is true for graphs satisfying Equation \ref{Equation1}, see \cite{MR2399379}.  It follows that $k\leq \omega(G)+1$ for any $(k|k)$-graph $G$ of order $n\leq 2k-2$, that is, $\omega(G)=k$ or $\omega(G)=k-1$.

\begin{enumerate}
\item[Case 1:] $\omega(G)=k$.

Let $H_1$ be a clique of $G$ and $H_2=G\setminus V(H_1)$. There is a set of $k$ edges from $V(H_1)$ and $V(H_2)$. Therefore, if $t=n-k\leq k-2$ is the order of $H_2$ and $m=(kt-k)/2$ is the number of edges in $H_2$, then \[m\leq\binom{t}{2}.\] 
We obtain that $k\leq t$, a contradiction.

\item[Case 2:] $\omega(G)=k-1$.

Let $H_1$ be a clique of $G$ and $H_2=G\setminus V(H_1)$. There is a set of $2(k-1)$ edges from $V(H_1)$ to $V(H_2)$. Therefore, if $t=n-(k-1)\leq k-1$ is the order of $H_2$ and $m=(kt-2(k-1))/2$ is the number of edges in $H_2$, then \[m\leq\binom{t}{2}.\] 
We obtain that $k\leq t+1$, hence, $k=t+1$ and $n$ has to be $2k-2$. Since every vertex $v$ in $V(H_2)$ has degree $k$ in $G$, $v$ has at least two neighbours in $H_1$. By symmetry, $G$ is the union of two complete graphs $K_{k-1}$ with the addition of two perfect matchings between them. Its complement is a $(k-3)$-regular bipartite graph. Any perfect matching of $G^c$ induce a $(k-1)$-coloring in $G$, a contradiction.
\end{enumerate}

We have the following results.

\begin{lemma}
For any $k\geq 3$,
\[2k-1\leq n(k|k)\leq c(k|k)\leq 2k.\]
\end{lemma}

If $k$ is odd then the order of any $k$-regular graph is even, therefore:

\begin{corollary}
For any $k\geq 3$ an odd number, $n(k|k)=c(k|k)=2k$.
\end{corollary}

We have that $C^c_7$ is the extremal $(4|4)$-graph. Next, assume that $k\geq 6$ is an even number and there exists a $(k|k)$-graph $G$ of $n=2k-1$ vertices. Owing to the fact that $\chi(G)\leq n-\alpha(G)+1$ where $\alpha(G)$ is the independence number of $G$, we get that $\alpha(G)\leq k$.

In \cite{MR2399379} was proved that the Reed's conjecture holds for graphs of order $n$ satisfying $\chi>\frac{n+3-\alpha}{2}$. In the case of the graph $G$, we have that \[\frac{n+3-\alpha(G)}{2}\leq \frac{k}{2}+1<k.\]
It follows that $\omega(G)\leq k \leq \omega(G)+1$. Newly, we have two cases:
\begin{enumerate}
\item[Case 1:] $\omega(G)=k$.

As we saw before, let $H_1$ be a clique of $G$ and $H_2=G\setminus V(H_1)$. There is a set of $k$ edges from $V(H_1)$ and $V(H_2)$. Therefore, if $t=k-1$ is the order of $H_2$ and $m=(kt-k)/2$ is the number of edges in $H_2$, then \[m\leq\binom{t}{2}.\] 
We obtain that $k\leq t$, a contradiction.

\item[Case 2:] $\omega(G)=k-1$.

In \cite{MR3218276} was proved that every graph satisfies \[\chi\leq\left\{ \omega,\Delta-1,\left\lceil \frac{15+\sqrt{48n+73}}{4}\right\rceil \right\}.\]
Hence, for the graph $G$ we have that $k\leq \left\lceil\frac{15+\sqrt{96k+25}}{4}\right\rceil.$ After some calculations we get that $k=6,8,10,12$, otherwise, $k> \left\lceil\frac{15+\sqrt{96k+25}}{4}\right\rceil.$
\end{enumerate}

Finally, we have the following theorem.
\begin{theorem}
For any $k\geq 3$ such that $k\notin\{4,6,8,10,12\}$, \[n(k|k)=c(k|k)=2k.\] Moreover, if $k=4$ then $n(k|k)=c(k|k)=2k-1$ and if $k\in\{6,8,10,12\}$ then \[2k-1\leq n(k|k)\leq c(k|k)\leq 2k.\]
\end{theorem}
We point out that if there exists an extremal $(k|k)$-graph $G$ of $2k-1$ vertices for $k\in\{6,8,10,12\}$, then $G$ has clique number $\omega=k-1$, a clique $H_1$ of order $\omega$ for which $G\setminus V(H_1)$ has $\frac{k}{2}-1$ edges, $G$ is Hamiltonian-connected and it has independence number $\alpha(G)$ such that $\alpha(G)\in\{k/4,\dots,k/2+1\}$, see \cite{MR3218276}.


\section{Non-Cayley constructions}\label{section3}

In this section we improve the upper bound of $n(r|\chi)$ given on Theorem \ref{theorem2} by exhibiting a construction of graphs not necessarily Cayley. We assume that $r$ is not a multiple of $\chi-1$, therefore $2\leq \chi \leq r$. Additionally, we show two more constructions which are tight for some values.

\subsection{Upper bound}

To begin with, take the Tur{\' a}n graph $T_{n,\chi}$, for $n=a\chi+b$, $0<b<\chi$ with $r=a(\chi-1)+b$ and the partition $(V_1,V_2,\dots,V_b,V_{b+1},\dots,V_\chi)$ such that $|V_i|=a+1$ if $1\leq i \leq b$ and $|V_i|=a$ if $b+1\leq i \leq \chi$. Every vertex in $V_i$ for $1\leq i \leq b$ has degree $r-1$ and every vertex in $V_i$ for $b+1\leq i \leq \chi$ has degree $r$.

Next, we define the graph $G_{n,\chi}$ as the graph formed by two copies $G_1$ and $G_2$ of $T_{n,\chi}$ with the addition of a matching between the vertices of degree $r-1$ of $G_1$ and the vertices of degree $r-1$ of $G_2$ in the natural way. In consequence, the graph $G_{n,\chi}$ is an $r$-regular graph of order $2n$ and chromatic number $\chi$. To obtain its chromatic number, suppose that $T_{n,\chi}$ has the vertex partition $V_i$, then the vertices of $V_i$ have the color $i$ in $G_1$ and the vertices of $V_i$ are colored $i+1$ mod $\chi$ in $G_2$. Hence $\chi = \chi (G_1) \leq \chi (G_{n,\chi}) \leq \chi$ and then $\chi(G_{n,\chi}) =\chi$.

\begin{theorem}\label{theorem3}
For $2\leq \chi \leq r+1 $, then  \[\left\lceil \frac{r\chi}{\chi-1}\right\rceil\leq n(r|\chi) \leq \min\left\{2\left\lfloor \frac{r\chi}{\chi-1}\right\rfloor ,\frac{r-b}{\chi-1}\chi (b+1) \right\},\]
where $\chi-1|r-b$ with $0\leq b <\chi$. 
\end{theorem}

\subsection{The graph $T^*_{n,\chi}$}

In this subsection we give a better construction for some values of $r$ and $\chi$. Consider the $(a\chi+b,\chi)$-Tur{\' a}n graph $T_{a\chi+b,\chi}$ such that $\chi>b\geq 0$ and partition $(V_1,\dots,V_{\chi-b},\dots,V_\chi)$ for $\chi \geq 3$, $|V_i|=a_i=a \geq 2$ with $i\in\{1,\dots,\chi-b\}$ and $|V_i|=a_i=a+1 \geq 3$ with $i\in\{\chi-b+1,\dots,\chi\}$.

We claim that $a$ is even or $\chi-b$ is even. To prove it, assume that $a$ and $\chi-b$ are odd. Hence, if $b$ is even, then $\chi$ is odd, $n=a\chi+b$ is odd and $r$ is odd, a contradiction. If $b$ is odd, then $\chi$ is even, $n=a\chi+b$ is odd and $r$ is odd, newly, a contradiction.

Now, we define the graph $T^*_{n,\chi}$ of regularity $r=a(\chi-1)+b-1$ as follows: If $\chi-b$ is even, the removal of a perfect matching between $X_i$ and $X_{i+1}$ for all $i\in\{1,3,\dots,\chi-b-1\}$ of $T_{n,\chi}$ produces $T^*_{n,\chi}$. If $\chi-b\geq 3$ is odd then $a$ is even, therefore, the removal of a perfect matching between $X_i$ and $X_{i+1}$ for all $i\in\{4,6\dots,\chi-b-1\}$ and a perfect matching between $V'_1$ and $V''_2$, $V'_2$ and $V''_3$, and $V'_3$ and $V''_1$ where $V_i\setminus V'_i=V''_i$ is a set of $a/2$ vertices for $i\in\{1,2,3\}$, of $T_{n,\chi}$ produces $T^*_{n,\chi}$.

The graphs $T^*_{n,\chi}$ improve the upper bound given in Theorem \ref{theorem3} for some numbers $n$ and $\chi$:
\[\frac{r\chi}{\chi-1}=a\chi+b-\frac{\chi-b}{\chi-1}\leq a\chi+b.\]
Hence, if $\frac{\chi-b}{\chi-1}<1$, the construction gives extremal graphs, that is, when \[1<b.\] 

\begin{theorem}\label{theorem4}
Let $\chi \geq 3$, $\chi\geq b > 1$ and $a\geq 2$. Then the graph $T^*_{a\chi+b,\chi}$ defined above is an extremal $(a(\chi-1)+b-1|\chi)$-graph when $\chi-b$ is even or $a>2$ is even.
\end{theorem}

\subsection{The graph $G_{a,c,t}$}

Consider the $(at,t)$-Tur{\' a}n graph $T_{at,t}$ with partition $(V_1,\dots,V_t)$. Now, we define the graph $G_{a,c,t}$ with $1\leq c<a$ as follows: consider two parts of $(V_1,\dots,V_t)$, e.g. $V_1$ and $V_2$, and $c$ vertices of these two parts $\{u_1,\dots,u_c\}\subseteq V_1$ and $\{v_1,\dots,v_c\}\subseteq V_2$. 

The removal of the edges $u_iv_j$ for $i,j\in\{1,\dots,c\}$ when $i\not=j$ (all the edges between $\{u_1,\dots,u_c\}$ and $\{v_1,\dots,v_c\}$ except for a matching) and the addition of the edges $u_iu_j$ and $v_iv_j$ for $i,j\in\{1,\dots,c\}$ when $i\not=j$ (all the edges between the vertices $u_i$ and all the edges between the vertices $v_i$) results in the graph $G_{a,c,t}$.

The graph $G_{a,c,t}$ is a $a(t-1)$-regular graph of order $at$. Its chromatic number is $t+c-1$ because the partition 
\[(V_1\setminus\{u_2,\dots,u_c\},V_2\setminus\{v_1,\dots,v_{c-1}\},V_2,\dots,V_t,\{u_2,v_1\},\dots,\{u_c,v_{c-1}\})\]
is a proper coloring with $t+c-1$ colors. Moreover, the graph $G_{a,c,t}$ has a clique of $t+c-1$ vertices, namely, the vertices $\{u_1,\dots,u_c,x_2,\dots,x_t\}$ where $x_i\in V_i$ for $i\in\{3,\dots,t\}$ and $x_2\in V_2\setminus\{v_1\dots,v_c\}$.

The graphs $G_{a,c,t}$ improve the upper bound given in Theorem \ref{theorem2}:
\[\frac{t+c-1}{t+c-2}a(t-1)=at-a\frac{c-1}{t+c-2}\leq at.\]
Hence, if $a\frac{c-1}{t+c-2}<1$, the construction gives extremal graphs, that is, when \[(a-1)(c-1)<t-1.\] 

\begin{theorem}\label{theorem5}
Let $a,t\geq 2$ and $a>c\geq 1$. The graph $G_{a,c,t}$ defined above is an extremal $(a(t-1)|at)$-graph when $(a-1)(c-1)<t-1$.
\end{theorem}


\section{Small values} \label{section4}

In this section we exhibit extremal $(r|\chi)$-graphs of small orders. These exclude the extremal graphs given before. Table \ref{tabla1} shows the extremal $(r|\chi)$-graphs for $2\leq r \leq 10$ and $2\leq \chi \leq 6$.

\begin{table}[!htbp]
\begin{center}
\begin{tabular}{|c|c|c|c|c|c|}
\hline 
$r\setminus\chi$ & $2$ & $3$ & $4$ & $5$ & $6$\tabularnewline
\hline 
$2$ & $T_{4,2}$ & $T_{3,3}$ & - & - & -\tabularnewline
\hline 
$3$ & $T_{6,2}$ & $C^c_6$ & $T_{4,4}$ & - & -\tabularnewline
\hline 
$4$ & $T_{8,2}$ & $T_{6,3}$ & $C_7^c$ & $T_{5,5}$ & -\tabularnewline
\hline 
\multirow{2}{*}{$5$} & \multirow{2}{*}{$T_{10,2}$} & \multirow{2}{*}{$G_{5,2,2}$} & $C_8^c,(2C_4)^c,$ & \multirow{2}{*}{$K_5\times K_2$} & \multirow{2}{*}{$T_{6,6}$}\tabularnewline
 &  &  & $(C_3\cup C_5)^c$ &  & \tabularnewline
\hline 
$6$ & $T_{12,2}$ & $T_{9,3}$ & $T_{8,4}$ & $C_{9}^c,(C_4\cup C_5)^c$ & ? \tabularnewline
\hline 
\multirow{2}{*}{$7$} & \multirow{2}{*}{$T_{14,2}$} & \multirow{2}{*}{$T^*_{12,3}$} & \multirow{2}{*}{$T^*_{10,4}$} & $C^c_{10},(C_4\cup C_6)^c$ & \multirow{2}{*}{$(2C_5)^c$}\tabularnewline
 &  &  &  & $(C_3\cup C_7)^c$ & \tabularnewline
\hline
\multirow{2}{*}{$8$} & \multirow{2}{*}{$T_{16,2}$} & \multirow{2}{*}{$T_{12,3}$} & \multirow{2}{*}{$G_{4,2,3}$} & \multirow{2}{*}{$T_{10,5}$} & $C^c_{11},(C_4\cup C_7)^c$\tabularnewline
 &  &  &  &  & $(C_5\cup C_6)^c$\tabularnewline
\hline 
 &  &  &  &  & $C^c_{12},(2C_6)^c,(3C_4)^c$ \tabularnewline
$9$ & $T_{18,2}$ & $T^{**}_{16,3}$ & $T_{12,4}$ & $T^*_{12,5}$ & $(C_3\cup C_4 \cup C_5)^c$\tabularnewline
 &  &  &  &  & $(C_3\cup C_9)^c$ \tabularnewline
\hline 
$10$ & $T_{20,2}$ & $T_{15,3}$ & $T^*_{14,4}$ & $T^*_{13,5}$ & $T_{12,6}$\tabularnewline
\hline 
\end{tabular}
\caption{\label{tabla1}Extremal $(r|\chi)$-graphs.}
\par
\end{center}
\end{table}

\subsection{Extremal $(5|3)$-graph}

Suppose that $G$ is an extremal $(5|3)$-graph of order $8$, i.e., its order equals the lower bound given in Theorem \ref{theorem2}. Then its complement is $2$ regular. That is, $G^c$ is $C_8$ or $C_5\cup C_3$ or $C_4\cup C_4$. By Theorem \ref{theorem6}, the complement of $C_8$ or $C_5\cup C_3$ or $C_4\cup C_4$ has chromatic number $4$. Since $G$ is $5$-regular, a $(5|3)$-graph of order $9$ does not exist and therefore $10$ is the best possible. The graph $G_{5,2,2}$ is an extremal $(5|3)$-graph with $10$ vertices.

\subsection{Extremal $(7|\chi)$-graphs for $\chi=3,6$}

Let $G$ be an extremal $(7|3)$-graph. Its order is at least $11$. Since its degree is odd, its order is at least $12$. The graph $T^*_{12,3}$ is an extremal $(7|3)$-graph.

Now, suppose that $G$ is an extremal $(7|6)$-graph. $G$ has at least $9$ vertices. Newly, because it has an odd regularity, $G$ has at least $10$ vertices. If this is the case, its complement is a $2$ regular graph. The graph $(2C_5)^c$ has chromatic number $6$. It is unique and it is Cayley.

\subsection{Extremal $(9|3)$-graph}

Any $(9|3)$-graph has $14$ vertices, i.e., its order equals the lower bound given in Theorem \ref{theorem2}. Suppose that there exist at least one of degree $14$. Let $(V_1,V_2,V_3)$ a partition by independent sets. Some of the parts, $V_1$, has at least five vertices. Since the graph is $9$-regular, $V_1$ has exactly $5$ vertices. The induced graph of $V_2$ and $V_3$ is a bipartite regular graph of an odd number of vertices, a contradiction. Then, any $(9|3)$-graph has at least $16$ vertices.

Consider the graph $T_{16,3}$ with partition $(U,V,W)$ and the sets partition are $U=\{u_1,u_2,u_3,u_4,u_5\}$, $V=\{v_1,v_2,v_3,v_4,v_5\}$, $W=\{w_1,w_2,w_3,w_4,w_5,w_6\}$. The removal of the edges \[\{w_1v_1,v_1u_1,u_1w_4,w_2v_2,v_2u_2,u_2w_5,w_3v_3,v_3u_3,u_3w_6,u_4v_4,v_4u_5,u_5v_5,v_5u_4\}\] is the graph $T^{**}_{16,3}$ which is the extremal $(9|3)$-graph.

\section*{Acknowledgment}

We thank Robert Jajcay for useful discussions.

C. Rubio-Montiel was partially supported by PAIDI grant 007/19.

The authors wish to thank the anonymous referees of this paper for their suggestions and remarks.


\bibliographystyle{amsplain}
\bibliography{biblio}
\end{document}